\documentclass[a4paper,10pt]{article}
\usepackage{amsmath,amsfonts,amssymb}
\usepackage{mathrsfs}
\usepackage{amsthm}
\usepackage{multirow}
\usepackage[latin1]{inputenc}
\usepackage[T1]{fontenc}
\usepackage{color}
\usepackage{dsfont}
\usepackage{enumerate}
\usepackage{graphicx}
\usepackage{fullpage}
\usepackage{float}
\usepackage[titletoc,toc,title]{appendix}
\usepackage{pgf,tikz}
\usepackage{mathrsfs}
\usetikzlibrary{arrows}

\newcommand{\dd}{\mathrm{d}}

\DeclareMathOperator{\sgn}{sgn}
\DeclareMathOperator{\dvol}{dvol}

\DeclareMathOperator{\Ric}{Ric}

\newcommand{\com}[1]{}

\DeclareMathOperator{\vol}{vol}
\newcommand{\sph}{\mathcal{S}}

\newcommand{\R}{\mathbb{R}}

\newcommand{\J}{\mathcal{J}}

\author{}
\title{A local Blaschke-Petkantschin formula in a Riemannian manifold}
\date\today
\newtheorem{thm} {Theorem}[section]

\newtheorem{lem} [thm] {Lemma}

\newtheorem{prop} [thm] {Proposition}

\begin{document}
\author{Aur\'elie Chapron\textsuperscript{1}}
\footnotetext{\textsuperscript{2} Laboratoire MODAL'X, EA 3454, Universit\'e Paris Nanterre, 200 avenue de la R\'e\-pu\-bli\-que, 92001 Nanterre, France. E-mail: {\tt achapron@parisnanterre.fr}}
\maketitle

\begin{abstract}
In this paper, we show a local Blaschke-Petkantschin formula for a Riemannian manifold. Namely, we compute the Jacobian determinant of the parametrization of $(n+1)$-tuples of the manifold by the center and the radius of their common circumscribed sphere as well as the $(n+1)$ directions characterizing the positions of the $n+1$ points on it. We deduce from it a more explicit two-term expansion when the radius tends to $0$. This formula contains a local correction with respect to the flat case which involves the Ricci curvatures in the $(n+1)$ directions.
\end{abstract}

\section{Introduction and result}
A Blaschke-Petkantschin formula is a rewriting of the $(m+1)$-fold product of the volume measure for $1\le m\le n$ which is based on a suitable geometric decomposition of a $m$-tuple of points, see the historical papers \cite{Bla35} and \cite{Pet35}. In particular, it provides the calculation of the Jacobian of the associated change of variables in an integral. Classically, the geometric decomposition consists in fixing the linear or affine subspace which contains all the points, integrate over all $m$-tuples in that subspace and then integrate over the Grassmannian of all subspaces, see \cite[Chapter 7]{Wei08}. There exists another kind of formulas of Blaschke-Petkantschin type with a spherical decomposition: the sphere containing all the points is fixed, we integrate over the positions of the points on the sphere and then integrate over the Grassmannian of all subspaces spanned by the sphere and over the radius and center of the sphere. 

This can be described as follows in the Euclidean space. For almost every set of points $x_0,\dots,x_n \in \R^n$, there exists a unique circumscribed ball. This makes it possible to define the change of variables $(x_0,\dots,x_n) \to (r,z,u_0,\dots,u_n)$, given by the relations 
\[ x_i = z+ru_i \]
where $z\in \R^n$ and $r>0$ are the circumcenter and radius respectively of the circumscribed ball of $x_0,\dots,x_n$
and where $u_0,\dots,u_n\in \sph^{n-1}$ indicate the respective positions of the points $x_0,\dots,x_n$ on the boundary of that ball, $\sph^{n-1}$ being the unit-sphere of $\R^n$. The calculation of the Jacobian of this change of variables dates back to Miles \cite[Formula (70)]{Mil70}. It provides the following equality of measures, see also \cite[Proposition 2.2.3]{Mol94}
for a more general statement with $(m+1)$ points, $1\le m\le n$:
\begin{equation}\label{eucl} \dd x_0\dots \dd x_n =  n! \Delta_n(u_0,u_1,\dots,u_n)r^{n^2-1}\dd r \dd z\dvol^{(\sph^{n-1})}(u_0)\dots \dvol^{(\sph^{n-1})}(u_n) \end{equation}
where $\Delta_n(u_0,u_1,\dots,u_n)$ is the $n$-dimensional Lebesgue measure of the simplex spanned by the unit vectors $u_0,u_1,\dots,u_n$ and where $\vol^{(\sph^{n-1})}$
is the spherical Lebesgue measure of ${\sph^{n-1}}$. In general, for any Riemannian manifold $M$, we denote by $\vol^{(M)}$ the measure associated with the Riemannian volume form of $M$. 

The formula \eqref{eucl} was extended by Miles to the case of the unit-sphere
in dimension two \cite[Theorem 3.2]{Mil71} then in any dimension \cite[Theorem 4]{Mil71bis}. Let ${\mathcal S}_{k}^{n}$ be the $n$-dimensional sphere centered at the origin and of radius $1/k$. Again, for almost all $(n+1)$-tuple of points in ${\mathcal S}_{k}^{n}$, we denote by $z$ and $r\in (0,\frac{\pi}{2})$ the center and radius respectively of the associated geodesic circumscribed ball and by $u_0,\dots,u_n$ the respective positions of the points on the boundary of that ball.
The equality of measures below is then a slight modification of Miles's formula when replacing $\sph^{n}$ with ${\mathcal S}_{k}^{n}$:
 
\begin{align}\label{sph} 
&\dvol^{(\sph_k^n)}(x_0)\dots \dvol^{(\sph_k^n)}(x_n)\nonumber\\&\hspace*{1cm}=  n! \Delta_n(u_0,u_1,\dots,u_n)\left(\frac{\sin(kr)}{k}\right)^{n^2-1}\dd r\dvol^{(\sph_k^n)}(z) \dvol^{({\sph^{n-1}})}(u_0)\dots \dvol^{({\sph^{n-1}})}(u_n).\end{align}

The aim of this paper is to extend the spherical Blaschke-Petkantschin formulas \eqref{eucl} and \eqref{sph} to a general Riemannian manifold $M$, i.e. decompose the product measure $\dvol^{(M)}(x_0)\dots \dvol^{(M)}(x_n)$ through a geometric decomposition based on the smallest geodesic circumscribed ball containing the $(n+1)$ points $x_0,\dots,x_n$. One of our goals is to use this formula for calculating mean values of Poisson-Voronoi tessellations in a manifold \cite{AurelieArticle}. Indeed, Blaschke-Petkantschin formulas are already a key tool to study the characteristics of such tessellations in the Euclidean case \cite{Mol94} and in the spherical case \cite{Mil71}. In the hyperbolic case, Isokawa used for $n=2$ \cite{Iso2000} and $n=3$ \cite{Iso00b} formulas of the flavour of Blaschke-Petkantschin even though they were not clearly stated. 

We start by completing the work for manifolds with constant sectional curvature. For $k>0$, let ${\mathcal H}_{k}^{n}$ be the $n$-dimensional hyperbolic space of curvature $-k^2$. The next proposition provides the spherical Blaschke-Petkantschin formula when $M={\mathcal H}_{k}^{n}$.
\begin{prop}\label{prop:hyperb} 
For every $k>0$ and $n\ge 1$, the following equality of measures is satisfied.
\begin{align}\label{hyp} 
&\dvol^{({\mathcal H}_k^n)}(x_0)\dots \dvol^{({\mathcal H}_k^n)}(x_n)\nonumber\\&\hspace*{1cm}=  n! \Delta_n(u_0,u_1,\dots,u_n)\left(\frac{\sinh(kr)}{k}\right)^{n^2-1}\dd r\dvol^{({\mathcal H}_k^n)}(z) \dvol^{({\sph^{n-1}})}(u_0)\dots \dvol^{({\sph^{n-1}})}(u_n).
\end{align}
 \end{prop}
 
Now let $M$ be a Riemannian manifold of dimension $n$ and $\vol^{(M)}$ be its associated Riemannian measure. For any $z\in M$, we denote by $T_zM$ the tangent space of $M$ at $z$, by $\langle\cdot,\cdot\rangle_z$ the associated scalar product and by $\exp_z$ the exponential map at $z$. Since the geometric transformation and the Jacobian calculation will be done in a vicinity of a fixed point $z$, we can assume that $M$ is a compact set even if it means replacing $M$ by a compact neighborhood of $z$. We recall that thanks to the Hopf-Rinow theorem, the compacity of $M$ guarantees that $M$ is geodesically complete, which implies that the exponential map $\exp_{z}$ is defined on the whole tangent space $T_{z}M$.  Let us now define the set 
\[ \mathbb{T}_nM = \{ (z,u_0,\dots,u_n) \text{ such that }z\in M, u_i \in T_z M \mbox{ with norm $1$},i=0 ,\dots,n \} \]
and let us introduce the function 
\[\Phi_n:\left\{\begin{array}{ll}~(0,\infty) \times \mathbb{T}_nM & \to M^{n+1}\\
		(r,z,u_0, \dots,u_n) &\mapsto (x_0,\dots,x_n)\end{array}\right.\]
where
\begin{equation}\label{change}
x_i = \exp_z(ru_i). 
\end{equation}

For fixed $(z,u_0,\dots,u_n)\in {\mathbb T}_nM$ and fixed $1\le i\le n$, we denote by $\gamma_i$ the geodesic $\gamma_i:t\mapsto \exp_z(tu_i)$ so that $x_i=\gamma_i(r)$. We consider an orthonormal basis of $T_zM$, $\mathcal{V}^{(i)}=\{v_1^{(i)},\dots,v_{n}^{(i)}\}$ where $v_1^{(i)} = u_i$. The orthonormal basis of $T_{x_i} M$ denoted by $\mathcal{V}^{(i)}(r)=\{v_0^{(i)}(r),\dots,v_{n-1}^{(i)}(r)\}$ is obtained by parallel transport of $\mathcal{V}^{(i)}$ along $\gamma_i$.

We denote by $\J_{\Phi_n}$ the Jacobian of $\Phi_n$, i.e. the function which satisfies the equality of measures
$$\dvol^{(M)}(x_0)\dots \dvol^{(M)}(x_n)=  |\J_{\Phi_n}(r,z,u_0,\dots,u_n)|\dd r\dvol^{(M)}(z) \dvol^{({\sph^{n-1}})}(u_0)\dots \dvol^{({\sph^{n-1}})}(u_n).$$ 
We provide a formula for $\J_{\Phi_n}$ in terms of several explicit Jacobi fields. To do so, we use a precise connection between Jacobi fields and derivatives of curves defined through the exponential map, see Lemma \ref{lem:exodocarmo} below. 
\begin{thm}\label{BP}
The Jacobian determinant $\J_{\Phi_n}$ of $\Phi_n$  satisfies
\begin{equation}\label{Jac} |\J_{\Phi_n}(z,r,u_0,u_1,\dots,u_n)| = n! \Delta_n(u_0,\dots,u_n) \prod_{i=0}^{n} |\det B^{(i)}|
\end{equation}
where for any $0\le i\le n$, $B^{(i)}=(B^{(i)}_{l,m})_{1\le l,m\le (n-1)}$,  is a $(n-1)\times(n-1)$ real matrix which satisfies for $l,m=1,\dots,n-1$, \begin{align*}
				B^{(i)}_{l,m}& = \langle v_{l+1}^{(i)}(r), \tilde{J}_{m+1}^{(i)}(r)\rangle_{\gamma_i(r)},
			\end{align*}
$\tilde{J}_{m}^{(i)}$ being the Jacobi field along $\gamma_i$ with $ \tilde{J}_m^{(i)}(0)=0$ and ${\tilde{J}_m^{(i)\prime}}(0)= v_m^{(i)}$

{
Moreover, when $r$ tends to 0, 
\begin{equation}\label{dev} |\J_{\Phi_n}(z,r,u_0,u_1,\dots,u_n)| = n! \Delta_n(u_0,u_1,\dots,u_n)\left( r^{n^2-1} - \frac{\sum_{i=0}^{n} \Ric_z(u_i)}6 r^{n^2+1} + o(r^{n^2+1}) \right) \end{equation}
 where  $\Ric_{z}(u_i)$ denotes the Ricci curvature at $z$ of $u_i$ and $f(r)=o(g(r))$ means that $\lim_{r\to0}\frac{f(r)}{g(r)}=0$.}
\end{thm}
 
	As expected, since the manifold can be approximated at first order by the tangent space at $z$, the first term of the expansion \eqref{dev} corresponds to the Euclidean case. Moreover, when $M$ has constant sectional curvature, \eqref{dev}  is consistent with both \eqref{sph} and \eqref{hyp} for small $r$. 
	Let us note that for $r$ small enough, the Jacobian given by \eqref{Jac} is different from zero almost everywhere which implies thanks to the inverse function theorem that $\Phi_n$ is a local $C^1$-diffeomorphism. 
	
	As emphasized earlier, the Euclidean spherical Blaschke-Petkantschin formula was extended to provide the geometric decomposition  of the $(m+1)$-fold product of the Lebesgue measure on $\R^n$, for any $m=1, \dots,n$. We recall below this general formula, see e.g. \cite[Proposition 2.2.3]{Mol94}. Let $G(n,m)$ denote the Grassmanian of $m$-dimensional space of $\R^{n}$ endowed with its Haar measure $\vol^{(G(n,m))}$ and let us consider the set 
	\[ \mathbb{G}(n,m)=\{ (L_m, u_0, \dots, u_m), L_r \in G(n,m), u_i \in L_m \mbox{ with norm 1}\}. \]Then the change of variables defined by
	\[
	  \left\{\begin{array}{ll}\R_+ \times \R^{n} \times \mathbb{G}(n,m) &\to (\R^n)^{m+1}\\
										(r,z,L_m,u_0,\dots,u_m) &\mapsto (x_0,\dots,x_m)
	\end{array}\right.\]
	with $x_i= z+ru_i$, satisfies the equality of measures
	    \[ \dd x_0\dots \dd x_m= c_m^{(n)}(m! \Delta_m(u_0,\dots,u_m))^{n-m+1} r^{nm-1} \dd r\dd z \dvol^{(G(n,m))}(L_m)\dvol^{({\sph^{n-1}})}(u_0)\dots \dvol^{({\sph^{n-1}})}(u_m) \]
	where $c_m^{(n)}$ is an explicit constant depending only on $n$ and $m$. 

Going back to our manifold setting, it seems delicate to extend Theorem \ref{BP} to the decomposition of an $(m+1)$-fold measure for any $m$. 
However, we are able to derive formulas of this type for the manifold $M$ in the special cases $m=1$ and $m=(n-1)$ in Propositions \ref{deux} and \ref{np} respectively. Indeed, in these particular cases, we can identify the Grassmanian with the unit sphere ${\sph^{n-1}}$ and this makes it possible to write expansions of the Jacobian determinant. 
	\begin{prop}[Case $m=1$] \label{deux}
	The Jacobian $\J_{\Phi_1}$ of the function 
	\[\Phi_1 :\left\{\begin{array}{ll}  \R_+ \times {\mathbb T}_1M &\to M^2\\
										(r,z,u) &\mapsto (x_0=\exp_z(ru), x_1=\exp_z(-ru))\end{array}\right. \]
	satisfies
	\begin{equation}\label{eq2} |\J_{\Phi_1}(r,z,u)|= 2^n ( r^{n-1} - \frac 23 \Ric_z(u)r^{n+1} + o(r^{n+1}) )\end{equation}
	
	\end{prop}
In particular, we notice that the first term of the expansion in \eqref{eq2} is exactly the Jacobian in the Euclidean case since $\Delta_1(u,-u)$ is equal to $2$. 

Let us define the set
$${\mathbb T}_n^{\bot}M=\{ (z,v,u_0,\dots,u_{n-1}), x\in M, v,u_0,\dots,u_{n-1}\in T_zM \mbox{ with norm $1$ and $v\bot \{u_0,\dots,u_{n-1}\}$} \}.$$
	\begin{prop}[Case $m=(n-1)$]\label{np}
	The Jacobian $\J_{\Phi_{n-1}}$ of the function 
	
	\[ \Phi_{n-1} : \left\{\begin{array}{ll}  \R_+ \times {\mathbb T}_n^{\bot}M &\to M^n\\
										(r,z,v,u_0,\dots,u_{n-1}) & \mapsto ( x_i=\exp_z(ru_i))_{i=0,\dots,n-1} \end{array}\right.\]

	satisfies
	\begin{align}\label{eqn} & |\J_{\Phi_{n-1}}(z,r,v,u_0,\dots,u_{n-1})|\nonumber\\&=((n-1)! \Delta_{n-1}(u_0, \dots, u_{n-1}) )^2\left( r^{n(n-1)-1} - r^{n(n-1)+1}( \sum_{i=0}^{n-1} \Ric_z^{v}(u_i) + \mathrm{Tr}( \Delta^{-1}K) ) +o( r^{n(n-1)+1}) \right)\end{align}
where 

\[\Delta=  \begin{pmatrix} 1& -u_0^{T} \\
																			\vdots &\vdots \\
																			1 & -u_{n-1}^{T}
																			\end{pmatrix}   \text{ and } K= \begin{pmatrix} \frac{K(u_0,v)}{2}& u_0^{T}\frac{K(u_0,v)}{6} \\
																			\vdots &\vdots \\
																			\frac{K(u_{n-1},v)}{2} & u_{n-1}^{T}\frac{K(u_{n-1},v)}{6}
																			\end{pmatrix}   \]
	
	\end{prop}
	
The paper is structured as follows: we start with some geometrical preliminaries in Section \ref{sec:prelim}. Section \ref{sec:partie1} is devoted to the  case $m=n$ with the proofs of Theorem \ref{BP} and Proposition \ref{prop:hyperb} which is a corollary of Theorem \ref{BP}. The particular cases $m=1$ and $m=(n-1)$ are derived in Sections \ref{sec:partie2} and \ref{sec:partie3} respectively. The paper ends with some concluding remarks. 

\section{Geometrical preliminaries} \label{sec:prelim}

In this section, we introduce some useful notation and we survey several fundamental definitions and results from the theory of Riemannian geometry. For more details, we refer the reader to the reference books such as \cite{Do92}, \cite{Lee97} and \cite{Ber03}.
\paragraph{Exponential map.} Let $x \in M$ and $v\in T_{x} M$. There exists a unique geodesic $\gamma_v$ such that $\gamma_v(0) =x$ and $\gamma'_v(0) =v$. The exponential map of $v$ at $x$, denoted by $\exp_{x}(v)$ is defined by the identity
\begin{equation}\label{exp}
\exp_{x}(v) = \gamma_v(1)
\end{equation}
For sake of simplicity, we omit in the notation of the exponential map the dependency on the manifold $M$ which should be implicit anyway.

\paragraph{Riemann curvature tensor and curvatures.}
Let us denote by $\mathcal{R}^{(M)}$ the Riemann curvature tensor of $M$ that is for $x\in M$ and for $u,v,w \in T_xM$, 
\[ \mathcal{R}^{(M)}_x(u,v)w = \nabla_{u}\nabla_v w- \nabla_v \nabla_u w- \nabla_{[u,v]}w ,\]
where $\nabla$ denote the Levi-Civita connection. 
Let $u$,$v$ be two unit vectors of the tangent space $T_{x}M$ with $\langle u,v \rangle_{x}=0$. The sectional curvature of the plane spanned by $u$ and $v$ is defined through the identity 
\[ K_{x}^{(M)}(u,v)= \langle v, \mathcal{R}^{(M)}_x(u,v)u \rangle_x. \]

Now let $u$ be a unit vector of $T_{x} M$ and let us extend it to an orthonormal basis $\{ u_1, \dots, u_{n-1},u \}$ of $T_{x} M$. The Ricci curvature of $M$ at $x$ in direction $u$, denoted by $\Ric_{x}^{(M)}(u)$ is defined by the identity
\begin{equation}
  \label{eq:defric}
 \Ric_{x}^{(M)}(u) = \sum_{i=1}^{n-1} K_{x}^{(M)}(u,u_i).
\end{equation}
Note that $\Ric_{x}^{(M)}(u)$ does not depend on the choice of the basis.

\paragraph{Jacobi fields.} A Jacobi field along a geodesic $\gamma$ is a vector field $J$ verifying the Jacobi equation 
\begin{equation}\label{eqJacobi} J''(t) = \mathcal{R}^{(M)}_{\gamma(t)}(\gamma'(t),J(t)) \gamma' \end{equation}
where the derivative of $J$ is understood in the sense of the covariant derivative with respect to the Levi-Civita connection. In particular, along $\gamma$, there exists a unique Jacobi field with given $J(0)$ and $J'(0)$.

We recall without proof the following general result which connects the derivative of the exponential map to the Jacobi fields \cite[p. 119]{Do92}. This is a key tool of the proofs of Theorem \ref{BP}, Proposition \ref{deux} and Proposition \ref{np}. 
\begin{lem}\label{lem:exodocarmo}
Let $\gamma$ be a geodesic, $c$ a curve on $M$ such that $c(0)=\gamma(0)$ and $V$ a vector field along $c$ such that $V(0)= \gamma'(0)$. Then the function $f(t,s) = \exp_{c(s)}(t V(s))$ satisfies
\begin{equation}\label{deriv}
\frac{ \partial f}{\partial s} (t,0) = J(t)
\end{equation}
where $J$ is the unique Jacobi field along $\gamma$ with $J(0)= c'(0)$ and $J'(0)=V'(0)$.   
\end{lem}

In the case of manifolds with constant sectional curvature, Jacobi fields have an exact expression \cite[p113]{Do92}. The particular case of the hyperbolic space $\mathcal{H}_k^n$, stated in the following lemma, is the main argument of the proof of Proposition \ref{hyp}. 
\begin{lem}\label{lem:JFhyper}
Let $\gamma$ be a geodesic of $\mathcal{H}_k^n$ and $V$ be a parallel vector field along $\gamma$ with norm $1$ such that $\langle \gamma'(t), V(t) \rangle_{\gamma(t)}=0$ and $\| V(t) \|=1$. Then the vector field defined by 
\begin{equation}
 J(t) = \frac{ \sinh(kt)}{k} V(t)
\end{equation}
is the unique Jacobi field of $\mathcal{H}_k^n$ along $\gamma$ which satisfies $J(0)=0$ and $J'(0)=V(0)$. 
\end{lem}
 
\paragraph{Parallel transport.} Let $V$ be a vector field along a curve $\gamma$. $V$ is called a parallel vector field if $V'(t) =0$ for all $t$, in the sense of the covariant derivative. Now, for any $u \in T_{\gamma(0)}$, there exists a unique parallel vector field $V$ along $\gamma$ such that $V(0)=u$, called the parallel transport of $u$ along $\gamma$. In this paper, $V(t)$ will be denoted by $u(t)$. 
 Note that the parallel transport is a linear isomorphism from $T_{\gamma(0)}M$ to $T_{\gamma(t)}M$ which preserves the scalar product. 

	\section{Proof of Theorem \ref{BP} and Proposition \ref{prop:hyperb}}\label{sec:partie1}
\subsection{Proof of \eqref{Jac}}
\label{subsec:Jacobien}

Let us fix $(z,u_0,\dots,u_n)\in {\mathbb T}_nM$.
We endow the tangent space $T_z M$ with an orthonormal basis $\{ e_1,\dots, e_n \} $ and write 
\[ u_i=   \sum_{j=1}^{n} u_i^j e_j .\] 
In order to write the Jacobian matrix of $\Phi_n$, we need to introduce well-adapted bases of the tangent spaces $T_{x_i}M$ for $i=0, \dots, n$. To this end, for each $i$ we consider an orthonormal basis of $T_zM$, $\mathcal{V}^{(i)}=\{v_1^{(i)},\dots,v_{n}^{(i)}\}$ where $v_1^{(i)} = u_i$. We consider the basis of $T_{x_i} M$, $\mathcal{V}^{(i)}(r)=\{v_1^{(i)}(r),\dots,v_{n}^{(i)}(r)\}$, obtained by parallel transport of $\mathcal{V}^{(i)}$ along $\gamma_i(t) =\exp_z(tu_i)$.\\
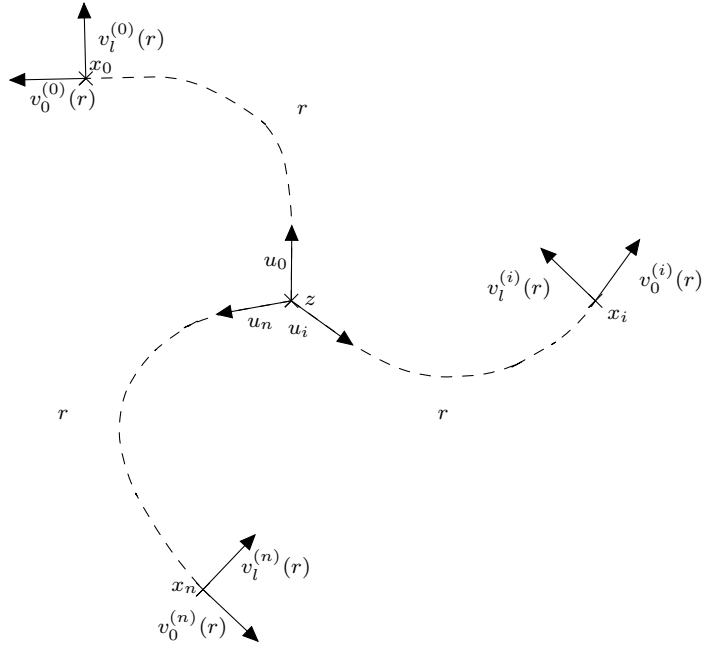
\begin{figure}[htbp]
	\centering

\begin{tikzpicture}[line cap=round,line join=round,>=triangle 45,x=1.0cm,y=1.0cm]
\clip(-6.,-5.) rectangle (6.,5.);

\draw [dash pattern=on 4pt off 4pt, rounded corners=10pt] (0.,0.)-- (-1.0953386919294494,-0.19696469112468673)-- (-1.533645371203449,-0.4210872360077298)-- (-2.0396453712034486,-0.8830872360077288)-- (-2.303645371203449,-1.5210872360077274)-- (-2.1936453712034485,-2.269087236007726)-- (-1.8856453712034489,-2.8630872360077246)-- (-1.423645371203449,-3.5450872360077232)-- (-1.1606681674018984,-3.827904048585847);
\draw [dash pattern=on 4pt off 4pt,rounded corners=10pt] (0.,0.)-- (1.0132746666185437,-0.7287831045637811)-- (1.7410261797457272,-1.0273478278980097)-- (2.7343546287965497,-0.9710872360077286)-- (3.174354628796549,-0.7290872360077292)-- (3.724354628796549,-0.39908723600772983)-- (4.,0.);
\draw [dash pattern=on 4pt off 4pt,rounded corners=10pt] (0.,0.)-- (0.008993075606364791,1.2774480272243458)-- (-0.14438896663782597,2.172176606982122)-- (-0.6045350933703982,2.5684135494462796)-- (-1.4225726520060824,2.9774323287641202)-- (-2.7023199040936117,2.949146848825859);
\draw [->] (0.,0.) -- (0.0071423691989552125,1.0145589609472216);
\draw [->] (0.,0.) -- (0.8215599759867002,-0.595331928030752);
\draw [->] (0.,0.) -- (-0.9982902145055506,-0.1811009284861391);
\draw [->] (-2.7023199040936117,2.949146848825859) -- (-3.7166564237771147,2.926734273952849);
\draw [->] (-2.7023199040936117,2.949146848825859) -- (-2.7263629691942977,3.9634460303776557);
\draw [->] (-1.1606681674018984,-3.827904048585847) -- (-0.4246453651114366,-4.526224414309614);
\draw [->] (-1.1606681674018984,-3.827904048585847) -- (-0.4641113035333426,-3.090212084366579);
\draw [->] (4.,0.) -- (4.59079504480661,0.8248285359443109);
\draw [->] (4.,0.) -- (3.2725750889092238,0.7072721522869793);

\begin{scriptsize}
\draw [color=black] (0.,0.)-- ++(-2.5pt,-2.5pt) -- ++(5.0pt,5.0pt) ++(-5.0pt,0) -- ++(5.0pt,-5.0pt);
\draw [color=black] (-1.1606681674018984,-3.827904048585847)-- ++(-2.5pt,-2.5pt) -- ++(5.0pt,5.0pt) ++(-5.0pt,0) -- ++(5.0pt,-5.0pt);
\draw [color=black] (4.,0.)-- ++(-2.5pt,-2.5pt) -- ++(5.0pt,5.0pt) ++(-5.0pt,0) -- ++(5.0pt,-5.0pt);
\draw [color=black] (-2.7023199040936117,2.949146848825859)-- ++(-2.5pt,-2.5pt) -- ++(5.0pt,5.0pt) ++(-5.0pt,0) -- ++(5.0pt,-5.0pt);
\end{scriptsize}
\node (a) at (0.25,0) { \footnotesize $z$};
\node (b) at (-1.4,-3.8) { \footnotesize $x_n$};
\node (c) at (-2.5,3.1) {\footnotesize $x_0$};
\node (d) at (4.3,-0.2) {\footnotesize $x_i$};
\node (e) at (0.15,2.54) {\footnotesize $r$};
\node (f) at (2,-1.5) {\footnotesize $r$};
\node (g) at (-3,-1.5) {\footnotesize $r$};
\node (h) at (-0.2,0.5) {\footnotesize $u_0$};
\node (i) at (-0.4,-0.3) {\footnotesize $u_n$};
\node (j) at (0.1,-0.4) {\footnotesize $u_i$};
\node (k) at (-2.1,3.5) {\footnotesize $v_{l}^{(0)}(r)$};
\node (l) at (-3, 2.7) {\footnotesize $v_{0}^{(0)}(r)$};
\node (m) at (-0.2,-3.5) {\footnotesize $v_{l}^{(n)}(r)$};
\node (n) at (-1.3,-4.3) {\footnotesize $v_{0}^{(n)}(r)$};
\node (o) at (3,0.2) {\footnotesize $v_{l}^{(i)}(r)$};
\node (p) at (5,0.3) {\footnotesize $v_{0}^{(i)}(r)$};

\end{tikzpicture}

		\caption{The function $\Phi_n$ and the bases $\mathcal{V}^{(i)}(r)$ }
\end{figure} 

\noindent {\it Step 1: derivatives with respect to $r$}.
Since $\frac{\partial x_i}{\partial r} = P_{z \to x_i}(u_i) = v_1^{(i)}(r)$, we notice that the column vector of the coordinates of $\frac{\partial x_i}{\partial r}$  in the basis $\mathcal{V}^{(i)}(r)$ is 
\begin{equation}
  \label{eq:teteEi}
\frac{\partial x_i}{\partial r}=\left(\begin{array}{c}1\\0\\\vdots\\0\end{array}\right)=\left(\begin{array}{c}1\\
\hline\\
{{\mathbf 0}_{n-1}}^T
\\~\\
\end{array}
\right) 
\end{equation}
where ${\mathbf 0}_{n-1}=(0,\dots,0)$. In the rest of the paper, we identify for sake of simplicity a vector (resp. a linear transformation) with its column vector in a  natural basis (resp its matrix with well-chosen natural bases). \\~\\
\noindent {\it Step 2: derivatives with respect to $z$}.
Let us consider the  submatrix with size $n\times n$ 
$\frac{\partial x_i}{\partial z}$. First, we notice that the entry
$\left(\frac{\partial x_i}{\partial z}\right)_{l,m}$
is the projection onto $v_l^{(i)}(r)$ of the derivative of $x_i$ with respect to $z$ in the direction $e_m$. Secondly, applying Lemma \ref{lem:exodocarmo} to $\gamma(t)=\gamma_i(t)=\exp_z(t u_i)$, $c(s)=\gamma(s)=\exp_z(se_m)$ and $V$ as the parallel transport of $u_i$ along $c$, we obtain that the derivative of $x_i$ with respect to $z$ in the direction $e_m$ is $J_m^{(i)}(r)$ where $J_m^{(i)}$ is the unique Jacobi field  along $\gamma_i$ such that $J_m^{(i)}(0)=c'(0)=e_m$ and ${J_m^{(i)\prime}}(0)=V'(0)=0$. We deduce from these two observations that  
\[
\left(\frac{\partial x_i}{\partial z}\right)_{l,m}
= \left( \frac{\partial x_i}{\partial z} \right)_{l,m} = \langle v_l^{(i)}(r), J_m^{(i)}(r) \rangle_{x_i}.\]
Now thanks to \cite[Chapter 5, Proposition 3.6]{Do92},  we can do the following calculation:
\begin{equation}
  \label{eq:Ai1ereligne}
 \langle v_1^{(i)}(r), J_m^{(i)}(r)\rangle_{x_i}= \langle u_i,{J_m^{(i)\prime}}(0) \rangle_{z} r +\langle u_i,J_m^{(i)}(0)\rangle_{z} = \langle u_i,e_m\rangle_{z}=u_i^{m}. 
\end{equation}
In particular, \eqref{eq:Ai1ereligne} shows that the first line of the submatrix $A^{(i)}$ is equal to $u_i^T$. In other words, we can rewrite $A^{(i)}$ as 
\begin{equation}
  \label{eq:teteAi}
\frac{\partial x_i}{\partial z}
= \left(\begin{array}{lll}& u_i^{T}& \\\hline\\ & \mbox{ n/a } & \\ & & \end{array}\right)
\end{equation}
where `n/a' only means that the submatrix of $A^{(i)}$ under its first line does not need to be explicit in the rest of the proof.\\~\\
\noindent {\it Step 3: derivatives with respect to $u_i$}.
The submatrix $
\frac{\partial x_i}{\partial u_i}$ has size $n\times (n-1)$. Again the entry 
$\left(\frac{\partial x_i}{\partial u_i}\right)_{l,(m-1)}$
for $1\le l\le n$ and $2\le m\le n$, 
is the projection onto $v_l^{(i)}(r)$ of the derivative of $x_i$ with respect to $u_i$ in the direction $v_{m}^{(i)}$. Moreover, applying Lemma \ref{lem:exodocarmo} to $\gamma=\gamma_i$, $c(s)=z$ and $V(s)=v_1^{(i)}+sv_m^{(i)}$, we get that the derivative of $x_i$ with respect to $u_i$ in the direction $v_{m}^{(i)}$ is $\tilde{J}_m^{(i)}(r)$ where $\tilde{J}_m^{(i)}$ is the unique Jacobi field along $\gamma_i$ such that $\tilde{J}_{m}^{(i)}(0)=c'(0)=0$ and ${\tilde{J}_m^{(i)\prime}}(0)=V'(0)=v_m^{(i)}$. Consequently, this implies that 
\[
\left(\frac{\partial x_i}{\partial u_i}\right)_{l,(m-1)}
= \left( \frac{\partial x_i}{\partial u_i} \right)_{l,m} = \langle v_l^{(i)}(r), \tilde{J}_m^{(i)}(r)\rangle_{x_i}. \]
An application of \cite[Chapter 5, Proposition 3.6]{Do92} shows that $\langle v_1^{(i)}(r), \tilde{J}_m^{(i)}(r)\rangle =0$ which means that $\tilde{J}_m^{(i)}$ is a normal Jacobi field. In particular, this implies that the first line of $B^{(i)}$ is identically equal to zero. In the sequel, we denote by $B^{(i)}$ the $(n-1)\times (n-1)$-matrix under the first line of 
$\frac{\partial x_i}{\partial u_i}$, i.e. such that
\begin{equation}
  \label{eq:teteBi}
\frac{\partial x_i}{\partial u_i}
=\left(\begin{array}{lll} & {\mathbf 0}_{n-1} & \\\hline\\ & B^{(i)} & \\ & & \end{array}\right).
\end{equation}
In particular, we get for every $2\le l,m\le n$,
\begin{equation}
  \label{eq:coeffMi}
B_{l-1,m-1}^{(i)}=\langle v_{l}^{(i)}(r), \tilde{J}_{m}^{(i)}(r)\rangle_{x_i}
\end{equation}

\noindent {\it Step 4: rewriting of the Jacobian determinant.}
The Jacobian determinant $\J_{\Phi_n}$ of $\Phi_n$ can be written as 
$$\J_{\Phi_n}(z,r,u_0,\dots,u_n) =\det \left(\begin{array}{cccccc}
\cline{3-3} 
\frac{\partial x_0}{\partial r}&\frac{\partial x_0}{\partial z}&\multicolumn{1}{|c|}{\frac{\partial x_0}{\partial u_0}}&~&\multirow{2}{0cm}{\text{\bf 0}}&~\\
\cline{3-3}
 & & & & &\\
\cline{4-4}
\frac{\partial x_1}{\partial r}&\frac{\partial x_1}{\partial z}&~&\multicolumn{1}{|c|}{\frac{\partial x_1}{\partial u_1}}&~~&~\\
\cline{4-4}
 & & & & &\\
\vdots&\vdots&~&\multirow{2}{0cm}{\text{\bf 0}}~~&\ddots&~\\
 & & & & &\\
\cline{6-6}
\frac{\partial x_n}{\partial r}&\frac{\partial x_n}{\partial z}&~&~&~&\multicolumn{1}{|c|}{\frac{\partial x_n}{\partial u_n}}\\
\cline{6-6}
\end{array}\right)$$
Combining this equality with \eqref{eq:teteEi}, \eqref{eq:teteAi} and \eqref{eq:teteBi}, we obtain that
\begin{equation}\label{jac2} \J_{\Phi_n}(z,r,u_0,\dots,u_n) =\det \left(\begin{array}{c|ccc|ccc|ccc|ccc} 
1& &u_0^T &  &&{\mathbf 0}&& &{\mathbf 0}& & &{\mathbf 0}&\\\hline&&&&&&&&&&&&\\
{\mathbf 0}&&\mbox{n/a}&&& B^{(0)}&& &{\mathbf 0}& &&{\mathbf 0}&\\&&&&&&&&&&&&\\\hline &&&&&&&&&&&&\\
\vdots&&\vdots&& &{\mathbf 0}&& 
&\ddots& &&{\mathbf 0}& \\ &&&&&&&&&&&&\\\hline
1&&u_n^T&&&{\mathbf 0}&& & {\mathbf 0}& &&{\mathbf 0}&\\\hline&&&&&&&&&&&&\\
{\mathbf 0}&&\mbox{n/a}&&&{\mathbf 0}& & &{\mathbf 0}& &&B^{(n)}&\\&&&&&&&&&&&&
\end{array}\right)\end{equation}
where for sake of simplicity, we have written ${\mathbf 0}$ for any zero matrix independently of its size. 

We can then apply a permutation of the lines so that the lines $(1|u_i^T|{\mathbf 0})$ appear in the first $(n+1)$ lines of a new matrix which is a block lower triangular matrix and which has the same determinant as the Jacobian determinant up to a possible minus sign:
\begin{equation}\label{jac3} |J_{\Phi_n}(z,r,u_0,\dots,u_n)| =|\det 
\left(\begin{array}{ccccc}
\cline{1-2}
\multicolumn{1}{|c}{1}&\multicolumn{1}{c|}{u_0^T}&~~&~~&~~\\
\multicolumn{1}{|c}{\vdots}&\multicolumn{1}{c|}{\vdots}&~~&{\bf 0}&~~\\
\multicolumn{1}{|c}{1}& \multicolumn{1}{c|}{u_n^{T}}& ~~&~~&~~\\
\cline{1-3}
0&~~&\multicolumn{1}{|c|}{B^{(0)}}&~~&~~\\
\cline{3-3}
\vdots&~~&~~&\ddots&~~\\
\cline{5-5}
0&~~&~~&~~&\multicolumn{1}{|c|}{B^{(n)}}\\
\cline{5-5}
\end{array}\right)|
\end{equation}

Using the fact that $n!\Delta_n(u_0,\cdots,u_n)=\det\left(\begin{array}{c|ccc}1 & & u_0^T &\\\vdots & & \vdots &\\ 1 & & u_n^T &\end{array}\right)$, we obtain \eqref{Jac}. 
\begin{flushright}
$\Box$
\end{flushright}
\subsection{Proof of \eqref{dev}}
We now derive from \eqref{Jac} the expansion \eqref{dev} for small values of $r$. This only requires to expand $\det B^{(i)}$. To do so, we expand the coefficients of $B^{(i)}$ given by \eqref{eq:coeffMi} in  Lemma \ref{diag} below.
\begin{lem}\label{diag}
The coefficients of $B^{(i)}$ satisfy, for $l,m =2, \dots, n$, 
\begin{align}
 B^{(i)}_{(m-1),(m-1)}= \langle v_m^{(i)}(r), \tilde{J}_m^{(i)}(r)\rangle_{x_i}&= r- \frac{K(u_i,v_m^{(i)})}6 r^3  + o(r^3) \label{eq:termesdiag}\\
 B^{(i)}_{(l-1),(m-1)} = \langle v_l^{(i)}(r), \tilde{J}_m^{(i)}(r) \rangle_{x_i}& = o(r^2), ~~l\neq m.\label{eq:termesnondiag}
\end{align}
where $K(u_i,v_m^{(i)})$ is the sectional curvature at $z$ of the vectors $u_i$ and $v_m^{(i)}$.
\end{lem}

\begin{proof}[Proof of Lemma \ref{diag}]

Let us define the function $f_{l,m}^{(i)}(r) = \langle v_l^{(i)}(r), \tilde{J}_m^{(i)}(r) \rangle_{x_i}$ for fixed $i,l$ and $m$. We only need to write the third order Taylor expansion of $f_{l,m}^{(i)}$ in the neighborhood of $0$.

Since $\tilde{J}_{m}^{(i)}(0)=0$,  we have 
\begin{equation}
  \label{eq:terme0}
f_{l,m}^{(i)}(0)=  \langle v_l^{(i)}, \tilde{J}_m^{(i)}(0) \rangle =0.
\end{equation}
Because $v_l^{(i)}(r)$ is a parallel transport, ${v_l^{(i)\prime}}(r)=0$ and the derivative ${f_{l,m}^{(i)\prime}}$ satisfies
$$ {f_{l,m}^{(i)\prime}}(r) =  \langle v_l^{(i)}(r), {\tilde{J}_m^{(i)\prime}}(r) \rangle_{x_i}.$$ Consequently, using $\tilde{J}_m^{(i)\prime}(0)=v_m^{(i)}$, we get
\begin{equation}
  \label{eq:terme1}
{f_{l,m}^{(i)\prime}}(0) = \langle v_l^{(i)}, v_m^{(i)} \rangle_{z} = \delta_{l,m}.
\end{equation}
The Jacobi field $\tilde{J}_m^{(i)}(r)$ along the curve $\gamma_i$ satisfies the Jacobi equation, see e.g. \cite[Chapter 5, \S 2]{Do92}:
$$ {\tilde{J}_m^{(i)\prime\prime}}(r)+\mathcal{R}(\tilde{J}_m^{(i)}(r),\gamma_i'(r))\gamma_i'(r)=0$$
where $\mathcal{R}$ is the Riemann curvature tensor. Combining this with the facts that $v_l^{(i)}(r)$ is a parallel transport and that $\gamma_i'(r)=v_1^{(i)}(r)$, we obtain
\[ {f_{l,m}^{(i)\prime\prime}}(r) =\langle v_l^{(i)}(r), {\tilde{J}_m^{(i)\prime\prime}}(r) \rangle_{x_i} =- \langle v_l^{(i)}(r), \mathcal{R}^{(M)}_{x_i}(\tilde{J}_m^{(i)}(r), v_1^{(i)}(r))v_1^{(i)}(r) \rangle_{x_i}\]
and in particular 
\begin{equation}
  \label{eq:terme2}
{f_{l,m}^{(i)\prime\prime}}(0) = 0.
\end{equation}
Finally, using again the fact that $v_l^{(i)}$ is a parallel transport, we show that the third derivative of ${f_{l,m}^{(i)}}$ is
\[ {f_{l,m}^{(i)\prime\prime\prime}}(r) = \langle v_l^{(i)}(r), {\tilde{J}_m^{(i)\prime\prime\prime}}(r) \rangle_{x_i}.\]
Thanks to \cite[p.115]{Do92} and the equality $\gamma_i'(r)=v_1^{(i)}(r)$, we notice that 
\begin{equation}\label{eq:derive3} {\tilde{J}_m^{(i)\prime\prime\prime}}(r)= -\mathcal{R}_{x_i}^{(M)}({\tilde{J}_m^{(i)\prime}}(r), v_1^{(i)}(r))v_1^{(i)}(r).\end{equation}
Consequently, using that ${\tilde{J}_m^{(i)\prime}}(0)=v_m^{(i)}$, we deduce from the two previous equalities that when $l=m$, 
\begin{equation}
  \label{eq:terme3}
{f_{m,m}^{(i)\prime\prime\prime}}(0) = -\langle v_m^{(i)}, \mathcal{R}^{(M)}_{z}(v_m^{(i)}, v_1^{(i)})v_1^{(i)}\rangle_{x_i}= -K_z^{(M)}(u_i,v_m^{(i)}).
\end{equation}
where the last equality comes from the definition of the sectional curvature, see e.g. \cite[Proposition 3.1]{Do92}.
The estimates \eqref{eq:termesdiag} and \eqref{eq:termesnondiag} follow now by \eqref{eq:terme0}, \eqref{eq:terme1}, \eqref{eq:terme2} and \eqref{eq:terme3} combined with Taylor's theorem applied to $f_{l.m}^{(i)}$ in the neighborhood of $0$.
\end{proof}

Going back to the proof of \eqref{dev}, let us show that the determinant of $B^{(i)}$ has the required expansion 
\begin{equation}\label{detM}
  \det B^{(i)}= r^{n-1} -\frac{\Ric_z^{(M)}(u_i)}6 r^{n+1} + o(r^{n+1}). 
\end{equation}
Indeed, we can rewrite the determinant as
\begin{align*} \det B^{(i)} &= \sum_{\sigma \in \mathfrak{S}_{n-1}} \sgn(\sigma) \prod_{m=1}^{n-1} B^{(i)}_{\sigma(m),m}\\
                &=  \prod_{m=2}^{n}  \langle v_m^{(i)}(r), \tilde{J}_m^{(i)}(r) \rangle_{x_i} +  \sum_{\sigma \in \mathfrak{S}_{n-1} \backslash \{ \mathrm{Id} \}} \sgn(\sigma) \prod_{m=2}^{n}  \langle v_{\sigma(m)}^{(i)}(r), \tilde{J}_m^{(i)}(r) \rangle_{x_i},
								\end{align*}
where $\mathfrak{S}_{n-1}$ denotes the set of permutations of $\{1,\cdots,(n-1)\}$ and $\sgn(\sigma)$ is the signature of the permutation $\sigma$. For sake of simplicity, we use the same notation $\sigma$ for a permutation of either the set $\{1,\cdots,(n-1)\}$ or the set $\{2,\cdots,n\}$.

It then follows from \eqref{eq:termesdiag} and \eqref{eq:defric} that
\begin{equation}\label{det1} \prod_{m=2}^{n}  \langle v_m^{(i)}(r), \tilde{J}_m^{(i)}(r) \rangle_{x_i} =  r^{n-1} -\frac{\sum_{m=2}^{n} K_z^{(M)}(u_i,v_m^{(i)})}6 r^{n+1} + o(r^{n+1})= r^{n-1} -\frac{\Ric_z^{(M)}(u_i)}6 r^{n+1} + o(r^{n+1}). \end{equation}
It remains to show that for all $\sigma \neq \mathrm{Id}$, 
\begin{equation}
  \label{eq:permnonid}
  \lim_{r\to 0} \frac{1}{r^{n+1}}\prod_{m=2}^{n}  \langle v_{\sigma(m)}^{(i)}(r), \tilde{J}_m^{(i)}(r) \rangle_{x_i} =0. 
\end{equation}
Since $\sigma \neq \mathrm{Id}$, at least two indices, say $2$ and $3$ satisfy $2\neq \sigma(2)$ and $3\neq \sigma(3)$. Let us rewrite the product as
\begin{align}\label{eq:produitdecomp}
 \frac{1}{r^{n+1}}\prod_{m=2}^{n}  \langle v_{\sigma(m)}^{(i)}(r), \tilde{J}_j^{(i)}(r) \rangle_{x_i} &= \frac{1}{r^{n+1}}  \langle v_{\sigma(2)}^{(i)}(r), \tilde{J}_2^{(i)}(r) \rangle_{x_i}  \langle v_{\sigma(3)}^{(i)}(r), \tilde{J}_3^{(i)}(r) \rangle_{x_i} 
\prod_{j=4}^{n}  \langle v_{\sigma(m)}^{(i)}(r), \tilde{J}_m^{(i)}(r) \rangle_{x_i}.
\end{align}
Now, thanks to \eqref{eq:termesnondiag}, the first two terms satisfy 
\begin{equation}
  \label{eq:deuxpremierstermes}
 \lim_{r\to 0} \frac{1}{r^2} \langle v_{\sigma(2)}^{(i)}(r), \tilde{J}_2^{(i)}(r) \rangle_{x_i}   =\lim_{r\to 0} \frac{1}{r^2} \langle v_{\sigma(3)}^{(i)}(r), \tilde{J}_3^{(i)}(r) \rangle_{x_i}  =0.
\end{equation}
Looking at both \eqref{eq:termesdiag} and \eqref{eq:termesnondiag}, we observe that the remaining terms behave like $r^{n-3}$ at most, i.e. there is a positive constant $C$ such that for $r$ small enough,
\begin{equation}
  \label{eq:autrestermes}
\prod_{j=4}^{n}  \langle v_{\sigma(m)}^{(i)}(r), \tilde{J}_m^{(i)}(r) \rangle_{x_i} \le Cr^{n-3}.
\end{equation}
Thus inserting \eqref{eq:deuxpremierstermes} and \eqref{eq:autrestermes} into \eqref{eq:produitdecomp}, we obtain \eqref{eq:permnonid} which, combined to \eqref{det1}, implies in turn \eqref{detM}. The expansion \eqref{dev} is now a direct consequence of both \eqref{Jac} and \eqref{detM}. 
\begin{flushright}
$\Box$
\end{flushright}
\subsection{Proof of \eqref{hyp}}
Recall that $\J_{\Phi_n}(z,r,u_0,u_1,\dots,u_n)$ denote the Jacobian determinant of the function $\Phi_n$. To prove \eqref{hyp} it suffices to prove that, when $M=\mathcal{H}_k^n$,  
\[ |\J_{\Phi_n}(z,r,u_0,u_1,\dots,u_n)| = n! \Delta_n(u_0,\dots,u_n) \left( \frac{ \sinh(kr)}{k}\right)^{n^2-1} .\] 
We proved in Section \ref{subsec:Jacobien} that in a general manifold, 
\begin{equation}\label{JacobienGene} |\J_{\Phi_n}(z,r,u_0,u_1,\dots,u_n)| = n! \Delta_n(u_0,\dots,u_n) \prod_{i=0}^{n}| \det B^{(i)}| \end{equation}
where $ B^{(i)}$ are $(n-1)\times (n-1)$ matrices with coefficients
\begin{equation}\label{eq:teteBiBIS}
 B^{(i)}_{l-1,m-1}= \langle v^{(i)}_l (r), \tilde{J}^{(i)}_m(r) \rangle_{x_i} 
\end{equation}
and $\tilde{J}^{(i)}_m(r)$ is the unique Jacobi field with $\tilde{J}^{(i)}_m(0)=0$ and $\tilde{J}^{(i)\prime}_m(0)= v^{(i)}_m$. Now Lemma \ref{lem:JFhyper} provide exact expression of these Jacobi fields. Indeed, 
 applying Lemma \ref{lem:JFhyper} to $\gamma=\gamma_i$ and $V(t)= v^{(i)}_m(t)$, we obtain 
\begin{equation}\label{eq:JFhyperbolique}
\tilde{J}^{(i)}_m(r) = \frac{ \sinh(kr)}{k} v^{(i)}_m(r).
\end{equation}
Inserting \eqref{eq:JFhyperbolique} into \eqref{eq:teteBiBIS}, it follows that
\begin{equation}\label{eq:newBi}
B^{(i)}_{l-1,m-1} =  \langle v^{(i)}_l (r), \frac{ \sinh(kr)}{k} v^{(i)}_m(r) \rangle_{x_i} = \frac{ \sinh(kr)}{k} \delta_{l,m}
\end{equation}
so that 
\begin{equation}\label{eq:BiId}
B^{(i)}=  \frac{ \sinh(kr)}{k} \mathrm{Id}_{(n-1)}
\end{equation}
where $ \mathrm{Id}_{(n-1)}$ denotes the identity matrix of size $(n-1)$. 
Combining \eqref{eq:BiId} and \eqref{JacobienGene}, we obtain 

\[ |\J_{\Phi_n}^{\mathcal{H}_k^n}(z,r,u_0,u_1,\dots,u_n)| = n! \Delta_n(u_0,\dots,u_n) \left( \frac{ \sinh(kr)}{k}\right)^{n^2-1} \]
which establishes the equality of measures \eqref{hyp}.
\begin{flushright}
$\Box$
\end{flushright}

\section{Proof of Proposition \ref{deux}}\label{sec:partie2}
Let us calculate the Jacobian determinant of the function $\Phi_1$ given in Proposition \ref{deux}. To do so, 
let us consider an orthonormal basis of $T_zM$, $\mathcal{V} =\{ v_1, \dots,v_{n} \}$ where $v_1=u$. We consider the bases $\mathcal{V}^{(i)}(r)=\{  v^{(i)}_1(r), \dots,v^{(i)}_{n}(r) \}$ obtained by parallel transport of the basis $\mathcal{V}$ along $\gamma_i(t)= \exp_z((-1)^itu)$.\\ 

\noindent {\it Step 1: derivatives with respect to $r$}.
As in the proof of \eqref{Jac},
the column vector 
of the coordinates of $\frac{\partial x_i}{\partial r}$  in the basis $\mathcal{V}^{(i)}(r)$ is 
\begin{equation}
  \label{eq:teteEi2}
\frac{\partial x_i}{\partial r}=\left(\begin{array}{c}(-1)^i\\
\hline\\
{{\mathbf 0}_{n-1}}^T
\\~\\
\end{array}
\right). 
\end{equation}
~\\
\noindent {\it Step 2: derivatives with respect to $z$}.
Again, we show that
the  submatrix with size $n\times n$,
$
\frac{\partial x_i}{\partial z}$ satisfies
\[
\left(\frac{\partial x_i}{\partial z}\right)_{l,m}
= \langle v_l^{(i)}(r), J_m^{(i)}(r) \rangle_{x_i}.\]
where $J_m^{(i)}$ is the unique Jacobi field  along $\gamma_i$ such that $J_m^{(i)}(0)=v_m$ and ${J_m^{(i)\prime}}(0)=0$. 
Using \cite[Chapter 5, Proposition 3.6]{Do92}, we obtain
\[  \langle   v_0^{(i)}(r), J_m^{(i)}(r)\rangle _{x_i}= \langle  u,{J_m^{(i)}}'(0)\rangle _{x_i} r +\langle  u,J_m^{(i)}(0)\rangle _{x_i} = \langle  u,v_m\rangle _{x_i}= \delta_{0,m},\]
which shows that the shape of 
$\frac{\partial x_i}{\partial z}$
is
\begin{equation}
   \label{eq:teteAi2}
\frac{\partial x_i}{\partial z}
= \left(\begin{array}{l|ll}1 & \mathbf{0}_{n-1} & \\\hline  & \multicolumn{1}{|c}{~}\\ \mbox{\small{n/a}} & \multicolumn{1}{|c}{A^{(i)}}& \\ &  \multicolumn{1}{|c}{~} &\end{array} \right).
 \end{equation}
where ${A}^{(i)}$ is a $(n-1)\times(n-1)$-matrix.\\~\\
\noindent {\it Step 3: derivatives with respect to $u_i$}.
$\frac{\partial x_i}{\partial u}$ with size $n\times (n-1)$ satisfies
\[
\left(\frac{\partial x_i}{\partial u}\right)_{l,m}
= \langle  v_l^{(i)}(r), \tilde{J}_m^{(i)}(r)\rangle _{x_i} ,l=0,\dots,n-1,m=1,\dots,n-1\]
 where $\tilde{J}_m^{(i)}$ is the unique Jacobi field along $\gamma_i$ such that $\tilde{J}_m^{(i)}(0) =0$ and ${\tilde{J}_{m}^{(i)\prime}} (0)= (-1)^i v_m$. Again, the first line of $B^{(i)}$ is identically equal to zero, which means that we can rewrite $B^{(i)}$ as \begin{equation}
  \label{eq:teteBi2}
\frac{\partial x_i}{\partial u}
=\left(\begin{array}{lll} & {\mathbf 0}_{n-1} & \\\hline\\ & {B}^{(i)} & \\ & & \end{array}\right).
\end{equation}
\noindent {\it Step 4: rewriting of the Jacobian determinant.}
We recall that the Jacobian determinant is 
\[\J_{\Phi_1}(z,r,u)= \det \begin{pmatrix} 
\frac{\partial x_0}{\partial r}    & \frac{\partial x_0}{\partial z}& \frac{\partial x_0}{\partial u}\\
&&\\
		\frac{\partial x_1}{\partial r}& \frac{\partial x_1}{\partial z}& \frac{\partial x_1}{\partial u}
\end{pmatrix} \]
Using \eqref{eq:teteEi2}, \eqref{eq:teteAi2}, \eqref{eq:teteBi2} and the exact same permutation of lines as for \eqref{jac3}, we can rewrite the matrix up to a possible minus sign in front of the determininant as the following block lower triangular matrix: 

\[ |\J_{\Phi_1}(z,r,u)|= \det \left( 
 \begin{array}{c@{}|c@{}c}\cline{1-1}\multicolumn{1}{|c|}{\begin{array}{rr} 1 & 1 \\
   -1 & 1
\end{array}} & {\mathbf 0} &~~{\mathbf 0} ~~\\
\cline{1-3}
	\multirow{2}{0cm} {\mbox{n/a}} 
& ~~{A}^{(0)} ~~& \multicolumn{1}{c|}{{B}^{(0)}}\\
 &  {A}^{(1)}  &  \multicolumn{1}{c|}{{B}^{(1)}}\\\cline{2-3}\end{array}
\right)
= 2 \times\det \begin{pmatrix} 
{A}^{(0)}& {B}^{(0)}\\
{A}^{(1)}&{B}^{(1)} \\
\end{pmatrix} = 2 \det D\] 
where $D=\begin{pmatrix} 
{A}^{(0)}& {B}^{(0)}\\
{A}^{(1)}&{B}^{(1)} \\
\end{pmatrix}$.\\

\noindent {\it Step 5: expansion of the coefficients of the determinant.} As in the proof of \eqref{dev}, the required expansion \eqref{eq2} of the determinant follows from the expansion of each of its coefficients. This is done in the following lemma.
\begin{lem}\label{coef}
The coefficients of ${A}^{(i)}$ and ${B}^{(i)}$ satisfy for $l,m=2,\dots,n$
\begin{align*}
 {A}^{(i)}_{(m-1),(m-1)}&= \langle v_m^{(i)}(r), J_m^{(i)}(r)\rangle _{x_i}= 1 - \frac{K_z^{(M)}(u,v_m)}{2}r^2 +o(r^2), \\
{A}^{(i)}_{(l-1),(m-1)}&= \langle v_l^{(i)}(r), J_m^{(i)}(r)\rangle _{x_i}= o(r), ~~l\neq m,\\
{B}^{(0)}_{(m-1),(m-1)}&= \langle v_m^{(0)}(r), \tilde{J}_{m}^{(0)}(r)\rangle _{x_i}= r- \frac{K_z^{(M)}(u,v_m)}{6}r^3 +o(r^3), \\
{B}^{(1)}_{(m-1),(m-1)}&=  \langle v_m^{(1)}(r), J_{m}^{(1)}(r)\rangle _{x_i}=-(r- \frac{K_z^{(M)}(u,v_m)}{6}r^3) +o(r^3), \\
{B}^{(i)}_{(l-1),(m-1)}&= \langle v_l^{(i)}(r), J_m^{(i)}(r)\rangle _{x_i}= o(r^2) ,~~l\neq m.
\end{align*}
\end{lem}

\begin{proof}[Proof of Lemma \ref{coef}] 
We omit the proof for the coefficients of ${B}^{(i)}$ as it is very similar if not identical to the proof of Lemma \ref{diag}. 

Let us show the first two expansions. To do so, we consider the function $f_{l,m}^{(i)}(r)=\langle v_l^{(i)}(r), J_m^{(i)}(r)\rangle _{x_i}$ and apply Taylor's theorem to it at the second order in the neighborhood of 0.

Since $J_m^{(i)}(0)=v_m$, we have
\begin{equation}
  \label{eq:termeconst}
  f_{l,m}^{(i)}(0)=\delta_{l,m}.
\end{equation}
Because $v_l^{(i)}$ is a parallel transport, its derivative is equal to zero and so
  \begin{equation}
    \label{eq:termeordre1}
    f_{l,m}^{(i)\prime}(0)=\langle v_l^{(i)}(0), J_m^{(i)\prime}(0)\rangle _{z}=0.
  \end{equation}
The required estimate of ${A}^{(i)}_{(l-1),(m-1)}$ for $l\ne m$ follows from \eqref{eq:termeconst} and \eqref{eq:termeordre1}.

We calculate now the second derivative of $f_{m,m}^{(i)}$ at 0: since $J_m^{(i)}$ satisfies the Jacobi equation, we get
\begin{equation}
  \label{eq:termeordre2}
f_{l,m}^{(i)\prime\prime}(0)=-\langle v_m^{(i)}(0), {\mathcal R}_z^{(M)}(J_m^{(i)}(0),u)u\rangle _{x_i}=-\langle v_m, {\mathcal R}_z^{(M)}(v_m,u)u\rangle _{z}=-K_z^{(M)}(u,v_m).  
\end{equation}
It remains to combine \eqref{eq:termeconst}, \eqref{eq:termeordre1} for $l=m$ and \eqref{eq:termeordre2} with Taylor's theorem to deduce the required expansion of ${A}^{(i)}_{(m-1),(m-1)}$.
\end{proof}

\noindent{\it Step 6: conclusion.} Let us write the determinant of $D$ as
\begin{align}\label{eq:decompdet2}
 \det D &= \sum_{\sigma \in \mathfrak{S}_{2(n-1)}} \sgn(\sigma) \prod_{j=1}^{2(n-1)} D_{\sigma(j),j}= \sum_{\sigma \in \mathfrak{S}} \sgn(\sigma) \prod_{j=1}^{2(n-1)} D_{\sigma(j),j} +
\sum_{\sigma \in \mathfrak{S}_{2(n-1)}\backslash \mathfrak{S}} \sgn(\sigma) \prod_{j=1}^{2(n-1)} D_{\sigma(j),j}
\end{align}
where $\mathfrak{S}= \{ \sigma \in \mathfrak{S}_{2(n-1)}, \forall j, D_{\sigma(j),j}=  {A}^{(i)}_{m,m} \text{ or } D_{\sigma(j),j}=  {B}^{(i)}_{m,m}, \text{ for some }i,m \}$. 

We expect the contribution of the permutations in $\mathfrak{S}$ to be dominant in the sum above.  
Indeed, using the expansions contained in Lemma \ref{coef}, the fact that the cardinality of $\mathfrak{S}$ is $2^{n-1}$ and the identity \eqref{eq:defric}, we get
\begin{align} \label{eq:partiedom2}\sum_{\sigma \in \mathfrak{S}} \sgn(\sigma) \prod_{j=1}^{2(n-1)} D_{\sigma(j),j} &= 2^{n-1} \prod_{m=2}^{n} ( 1 - \frac{K(u,v_m)}{2}r^2 +o(r^2))( r- \frac{K_z^{(M)}(u,v_m)}{6}r^3 +o(r^3) ) \nonumber\\&= 2^{n-1} (r^{n-1} - \frac{2}{3} \Ric_z^{(M)}(u)r^{n+1} + o(r^{n+1})) \end{align}
Moreover, for any $\sigma\in \mathfrak{S}_{2(n-1)}\backslash \mathfrak{S}$, there exist at least two $j$ such that $D_{\sigma(j),j}$ is equal to ${A}^{(i)}_{m,l}$ or ${B}^{(i)}_{m,l}$ for different $m$ and $l$. Because of Lemma \ref{coef}, this means that the contribution of the product is at least of order $r^{n-2}\times r^2\times r^2$. Consequently, this yields that
\begin{equation}
  \label{eq:partienondom2}
 \lim_{r\to 0}\frac{1}{r^{n+1}}\sum_{\sigma \in \mathfrak{S}_{2(n-1)}\backslash \mathfrak{S}} \sgn(\sigma) \prod_{j=1}^{2(n-1)} D_{\sigma(j)j} =0.
\end{equation}
Inserting \eqref{eq:partiedom2} and \eqref{eq:partienondom2} into \eqref{eq:decompdet2}, we get the required expansion \eqref{eq2}.
\begin{flushright}
$\Box$
\end{flushright}
\section{Proof of Proposition \ref{np}}\label{sec:partie3}
Let us calculate the Jacobian determinant of the function $\Phi_{n-1}$ given in Proposition \ref{np}. To do so, 
we consider an orthonormal basis of $T_zM$, $\mathcal{V} =\{ v_1, \dots,v_{n} \}$ where $v_n=v$. In a very similar way to the proof of Theorem \ref{BP}, we write each vector $u_i$, $0\le i\le (n-1)$ in this basis, i.e.
\[ u_i=   \sum_{j=1}^{n} u_i^j v_j .\]
In particular, we notice that $u_i^n=0$ for every $0\le i\le (n-1)$. For sake of simplicity, we will use in the rest of the paper the notation $u_i$ for the $(n-1)$-dimensional line vector $(u_1^1,\dots,u_1^{n-1})$.

Then for each $i$, we introduce an orthonormal basis of $T_zM$ $\mathcal{V}^{(i)}=\{  v^{(i)}_1, \dots,v^{(i)}_{n} \}$ where $v_1^{(i)}=u_i$ and $v_n^{(i)}=v$, and we do the parallel transportation of this basis $\mathcal{V}^{(i)}(r)=\{  v^{(i)}_1(r), \dots,v^{(i)}_{n}(r) \}$ along the curve $\gamma_i(t)= \exp_z(tu_i)$.\\ 

\noindent {\it Step 1: derivatives with respect to $r$}.
Again,
the column vector 
of the coordinates of $\frac{\partial x_i}{\partial r}$  in the basis $\mathcal{V}^{(i)}(r)$ is 
\begin{equation}
  \label{eq:teteEi3}
\frac{\partial x_i}{\partial r}
=\left(\begin{array}{c}1\\
\hline\\
{{\mathbf 0}_{n-1}}^T
\\~\\
\end{array}
\right). 
\end{equation}
\noindent{\it Step 2: derivatives with respect to $z$.} We consider the submatrix $
\frac{\partial x_i}{\partial z}$. We recall that for $1\le l,m\le n$ 
$$
\left(\frac{\partial x_i}{\partial z}\right)_{l,m}
=\langle v_{l}^{(i)}(r), J_m^{(i)}(r) \rangle_{x_i},$$
where $J_m^{(i)}$ is the unique Jacobi field  along $\gamma_i$ such that $J_m^{(i)}(0)=v_m$ and ${J_m^{(i)\prime}}(0)=0$.\\ 
Arguments similar to Step 2 of the proof of \eqref{Jac} show that 
$\frac{\partial x_i}{\partial z}$
is equal to
\begin{equation}
  \label{eq:teteAi3}
\frac{\partial x_i}{\partial z}
= \left(\begin{array}{lll|l} & u_i^{T} & & \multicolumn{1}{|c}{0}
\\\hline
\multirow{2}{0cm}{ }& \multirow{2}{0cm}{\mbox{n/a}} & \multirow{2}{0cm}{}& \multirow{2}{0cm}{${A}^{(i)}$}\\ & & &
\end{array}\,\right)
\end{equation}
where ${A}^{(i)}$ is a $(n-1)\times 1$-column vector with transpose 
$$\left({{A}}^{(i)}\right)^T=(\langle v_{2}^{(i)}(r), J_n^{(i)}(r) \rangle_{x_i},\dots, \langle v_{n}^{(i)}(r), J_n^{(i)}(r) \rangle_{x_i}).
$$
\noindent {\it Step 3: derivatives with respect to $u_i$}.
We first recall that $u_i$ is a vector from the $(n-1)$-dimensional subspace $\{v\}^\perp$ of $T_zM$. The submatrix $
\frac{\partial x_i}{\partial u_i}$ is then equal to
\begin{equation}
\label{eq:teteBi3}
\frac{\partial x_i}{\partial u_i}
=\left(\begin{array}{lll} & {\mathbf 0}_{n-1} & \\\hline\\ & {B}^{(i)} & \\ & & \end{array}\right).
\end{equation}
where ${B}^{(i)}$ is the $(n-1)\times(n-2)$-matrix such that for every $1\le l \le (n-1)$ and $1\le m\le (n-2)$,
\[{B}^{(i)}_{l,m} = \langle  v_{l+1}^{(i)}(r), \tilde{J}_{m+1}^{(i)}(r)\rangle _{x_i},\]
$\tilde{J}_m^{(i)}$ being unique Jacobi field along $\gamma_i$ such that $\tilde{J}_m^{(i)}(0) =0$ and ${\tilde{J}_{m}^{(i)\prime}} (0)= v_m^{(i)}$.\\~\\
 
\noindent {\it Step 4: derivatives with respect to $v$}. 
 When comparing to the proof of \eqref{Jac}, we observe that this step is new. The submatrix $
\frac{\partial x_i}{\partial v}$ has size $n\times (n-1)$ and $\left( \frac{\partial x_i}{\partial v} \right)_{l,m} $ is the projection onto $v_l^{(i)}(r)$ of the derivative of $x_i$ with respect to $v$ in the direction $v_m$. Let $m,i$ be fixed, $s>0$ and let us define the new basis of $T_zM$, $\mathcal{V}(s) = \{ v_1(s),\dots, v_n(s)=v(s) \}$ by 
\begin{align*}
 v(s)=v_n(s) &:=\frac{ v + sv_m }{\|  v + sv_m\| },\\
v_m(s) &:= \frac{v_m-sv}{ \| v_m-sv \|}, \\
v_j(s)& :=v_j, j\neq m, j\neq n.
\end{align*}
In this new basis, we define the vector $u_i^{s}$ by
\begin{align*}
 \tilde{u_i}({s}) &= \sum_{j=1}^{n-1} u_i^{j}v_j(s) \\
 & = u_i + u_i^{m} ( v_m(s)-v_m ) \\
& =u_i + u_i^m \left( \frac{v_m-sv}{ \| v_m-sv \|} -v_m \right).
  \end{align*}
Now, the derivative of $x_i$ with respect to $v$ in the direction $v_m$, is obtained by applying Lemma \ref{lem:exodocarmo} with $\gamma =\gamma_i$, $c(s) =z$ and $V(s)=\tilde{u_i}({s})$ that is 
  \begin{equation}\label{deriveV}
	 \frac{\partial x_i}{\partial v_m} = \bar{J}_{m}^{(i)}(r)
	\end{equation}
 where $\bar{J}_{m}^{(i)}$ is the unique Jacobi field along $\gamma_i$ such that $\bar{J}_{m}^{(i)}(0) =0$ and ${\bar{J}_{m}^{(i)\prime}} (0)=V'(0)$. We have
\begin{align*}
V'(s) & = u_i^m \left( \frac{-v}{\| v_m-sv\|} - \frac{(v_m-sv)\langle -v,v_m-sv \rangle_z}{\| v_m-sv\|^3} \right)
\end{align*} 
so $V'(0) =  - u_i^m v$. This, with  \eqref{deriveV}, implies that
 , for every $1\le l,m\le (n-1)$, 
\[\left( \frac{\partial x_i}{\partial v} \right)_{l,m} = \langle v_{l}^{(i)}(r), \bar{J}_{m}^{(i)}(r) \rangle_{x_i}, \]
 where $\bar{J}_{m}^{(i)}$ is the unique Jacobi field along $\gamma_i$ such that $\bar{J}_{m}^{(i)}(0) =0$ and ${\bar{J}_{m}^{(i)\prime}} (0)=-u_i^m v$.
As for $\frac{\partial x_i}{\partial u_i}$, the first line of  $\frac{\partial x_i}{\partial v}$ is identically equal to zero. Consequently, we get
\begin{equation}
  \label{eq:teteCi3}
  \frac{\partial x_i}{\partial v}=\left(\begin{array}{lll} & {\mathbf 0}_{n-1} & \\\hline\\ & {C}^{(i)} & \\ & & \end{array}\right)
\end{equation}
where $C^{(i)}$ is a $(n-1)\times(n-1)$-matrix satisfying $C^{(i)}_{l,m}=\langle v_{l+1}^{(i)}(r), \bar{J}_{m}^{(i)}(r) \rangle_{x_i}$.\\~\\
\noindent {\it Step 5: rewriting of the Jacobian determinant.}
Thus, the Jacobian determinant of $\Phi_{n-1}$, $\J_{\Phi_{n-1}}$ is \\
\begin{equation}\label{np2} \J_{\Phi_{n-1}}(z,r,v,u_0,\dots,u_{n-1}) =\det \left(\begin{array}{ccccccc} 
\cline{4-4}
\frac{\partial x_0}{\partial r}&\frac{\partial x_0}{\partial z}&\frac{\partial x_0}{\partial v}&\multicolumn{1}{|c|}{\frac{\partial x_0}{\partial u_0}}&~&~&~\\
\cline{4-4}
&&&&&\\
\cline{5-5}
\frac{\partial x_1}{\partial r}&\frac{\partial x_1}{\partial z}&\frac{\partial x_1}{\partial v}&~&\multicolumn{1}{|c|}{\frac{\partial x_1}{\partial u_1}}&~~ &\text{\bf{0}}\\
\cline{5-5}
&&&&&\\
\vdots&\vdots&\vdots&\text{\bf{0}}~~& &\ddots&~\\
 &&&&&\\
\cline{7-7}
\frac{\partial x_{n-1}}{\partial r}&\frac{\partial x_{n-1}}{\partial z}&\frac{\partial x_{n-1}}{\partial v}&~&~&~&\multicolumn{1}{|c|}{\frac{\partial x_{n-1}}{\partial u_{n-1}}}\\
\cline{7-7}
\end{array}\right).\end{equation}
Considering \eqref{eq:teteEi3}, \eqref{eq:teteAi3}, \eqref{eq:teteBi3} and \eqref{eq:teteCi3}, we can then apply the same permutation of the lines as in the proof of \eqref{Jac} so that the determinant is up to a possible minus sign equal to the determinant 
of a new matrix which is a block lower triangular matrix:
\begin{eqnarray}
\label{jacFin}
|\J_{\Phi_{n-1}}(z,r,v,u_0,\dots,u_{n-1})| &=
|\det \left( \begin{array}{cc|ccccc|}
\cline{1-2}
\multicolumn{1}{|c}{1} & \multicolumn{1}{c|}{u_0^T} & &&&&\multicolumn{1}{c}{}\\
\multicolumn{1}{|c}{\vdots} & \multicolumn{1}{c|}{\vdots} & \multicolumn{5}{c}{{\bf{0}}}\\
\multicolumn{1}{|c}{1} & \multicolumn{1}{c|}{u_{n-1}^T} & &&&&\multicolumn{1}{c}{}\\
\hline &&&&&&\multicolumn{1}{c|}{}\\\cline{5-5}
 & & \multicolumn{1}{|c}{A^{(0)}} & C^{(0)} & \multicolumn{1}{|c|}{B^{(0)}} &  & \multirow{2}{0cm}{\bf{0}}\\
\cline{5-5}
\multicolumn{2}{c}{\mbox{n/a}} & \multicolumn{1}{|c}{\vdots} & \vdots &  \multirow{2}{0cm}{\text{\bf{0}}} & \ddots &\\
\cline{7-7}
 & & \multicolumn{1}{|c}{A^{(n-1)}} & C^{(n-1)} & 

&  &\multicolumn{1}{|c|}{B^{(n-1)}}\\
\cline{7-7} &&&&&&\\
\cline{3-7}
\end{array} \right)|
\nonumber\\ ~\nonumber\\
& \hspace*{-3.6cm}= (n-1)! \Delta_{n-1}(u_0,\dots,u_{n-1}) | \det D |
\end{eqnarray}
where $D$ is the $n(n-1)\times n(n-1)-$matrix given by
\[ D = \begin{pmatrix}  {A}^{(0)} & {C}^{(0)}& {B}^{(0)} & &\multirow{2}{0cm}{\bf{0}}\\
																\vdots&         \vdots&\multirow{2}{0cm}{\bf{0}} & \ddots & \\
												{A}^{(n-1)} & {C}^{(n-1)}& & &{B}^{(n-1)}
			 \end{pmatrix}. \]
\noindent {\it Step 6: expansion of the coefficients of the determinant.} 
In Lemma \ref{lem:coeffpresde0} below, we provide estimates for the coefficients of the matrix $D$ when $r\to 0$.
\begin{lem}\label{lem:coeffpresde0}
The coefficients of $D$ satisfy for $l=2,\dots,n$, $m=2,\dots,(n-1)$
\begin{align}
A^{(i)}_{(n-1),1}&= \langle v_{n}^{(i)}(r), J_{{n}}^{(i)}(r)\rangle_{x_i} = 1-\frac{K_z^{(M)}(u_i,v)}{2}r^2 +o(r^2), \label{coefA1}\\
A^{(i)}_{(l-1),1}&= \langle v_{l}^{(i)}(r), J_{{n}}^{(i)}(r)\rangle_{x_i} = o(r), ~~l\neq n,\label{coefA2}\\
B^{(i)}_{(m-1),(m-1)}&=  \langle v_m^{(i)}(r), \tilde{J}_m^{(i)}(r)\rangle_{x_i} = r-\frac{K_z^{(M)}(u_i,v^{(i)}_m)}{6}r^3 +o(r^3),\label{coefB1}\\
B^{(i)}_{(l-1),(m-1)}&=  \langle v_l^{(i)}(r), \tilde{J}_m^{(i)}(r)\rangle_{x_i} = o(r^2), ~~l\neq m,\label{coefB2}\\
C^{(i)}_{(n-1),m}&= \langle v_n^{(i)}(r), \bar{J}_{m}^{(i)}(r)\rangle_{x_i}= -u^i_m( r-\frac{K_z^{(M)}(u_i,v)}{6}r^3) +o(r^3) ,\label{coefC1}\\
C^{(i)}_{(l-1),m}&= \langle v_l^{(i)}(r), \bar{J}_{m}^{(i)}(r)\rangle_{x_i} = o(r^2),~~l\neq n.\label{coefC2}
\end{align}

\end{lem}
\begin{proof} 
We only provide a proof for the last two expansions since the proof for the coefficients of $A^{(i)}$ and $B^{(i)}$ is very similar to the proofs of Lemmas \ref{diag} and \ref{coef}. Let us define the function $f^{(i)}_{l,m} (r) =  \langle v_l^{(i)}(r), \bar{J}_{m}^{(i)}(r)\rangle_{x_i}$ and write its third order Taylor expansion. To this end, let us compute the successive derivatives of $f^{(i)}_{l,m} $ at $0$. Since $\bar{J}_{m}^{(i)}(0)=0$, 
\begin{equation}\label{ordre0}
f^{(i)}_{l,m} (0) =0. 
\end{equation}
Because $v_l^{(i)}(r)$ is a parallel transport $v_l^{(i)\prime}(r)=0$ then 
\begin{equation}\label{ordre1}
f^{(i)\prime}_{l,m} (0) = \langle v_l^{(i)}(0), \bar{J}_{m}^{(i)\prime}(0)\rangle_{z} = -u^i_m \delta_{l,n}. 
\end{equation}
The Jacobi field $\bar{J}_{m}^{(i)}$ satisfies the Jacobi equation thus 
\begin{equation}\label{ordre2}
f^{(i)\prime\prime}_{l,m} (0) =  \langle v_l^{(i)}(0), \bar{J}_{m}^{(i)\prime\prime}(0)\rangle_{z} = -\langle v_l^{(i)}(0),\mathcal{R}( \bar{J}_{m}^{(i)}(0), v^{(i)}_1)v^{(i)}_1\rangle_{z}= 0. 
\end{equation} 
Now, \eqref{ordre0}, \eqref{ordre1} and \eqref{ordre2} with Taylor's theorem imply the expansion of $C^{(i)}_{(l-1),m}$, for $l\neq n$. 
For the case $l=n$, let us compute the third derivative of $f^{(i)}_{l,m} $. As in the proof of Lemma \ref{diag}, the identity \eqref{eq:derive3} is satisfied.
This with ${\bar{J}_m^{(i)\prime}}(0)= -u_m^i v$ yields
\begin{align}
f^{(i)\prime\prime\prime}_{n,m} (0) &=  \langle v_n^{(i)}(0), \bar{J}_{m}^{(i)\prime\prime\prime}(0)\rangle_{z} \notag \\
                                   & = -   \langle v_n^{(i)}(0),\mathcal{R}_z^{(M)}({\bar{J}_m^{(i)\prime}}(0), v_1^{(i)}(0))v_1^{(i)}(0) \rangle_{z} \notag \\ 
																	& =  u^i_m \langle v_n^{(i)}(0),\mathcal{R}_z^{(M)}(v_n^{(i)}(0), v_1^{(i)}(0))v_1^{(i)}(0)\rangle_{z} \notag \\
																	& = u^i_m K_z^{(M)}(v,u_i) \label{ordre3}
\end{align}
Combining \eqref{ordre0}, \eqref{ordre1} and \eqref{ordre2} for $l=n$ with \eqref{ordre3} and Taylor's theorem, we obtain the expansion of $C^{(i)}_{(n-1),m}$.  
\end{proof}

\noindent {\it Step 7: expansion of the determinant.} 
To find the expansion of $\det D$, we identify the dominant terms involved in the determinant, that is, we isolate the permutations providing the minimal power of $r$ in the final expansion. To this end, let us define the  set of permutations 
\begin{align*}
 \mathscr{S}& = \{\sigma \in \mathfrak{S}_{n(n-1)}, \\
& D_{\sigma(1),1} = A^{(i)}_{(n-1),1} \text{ for some }i=0,\dots,(n-1),\\
&  D_{\sigma(j),j} = C^{(\tau(j))}_{(n-1),m} \text{ for }j=2,\dots,n, \text{ and }\tau :\{ 2,\dots,n\} \to \{0,\dots,(n-1) \}\backslash \{i\} \text{ is some bijection,}\\ 
& D_{\sigma(j),j} = B^{(k)}_{m,m} \text{ for }j=n+k(n-2)+m\text{ where }k=0,\dots,(n-1),m=1,\dots,(n-2) \} .\end{align*}
Since only the integer $i$ and the bijection $\tau$ have to be chosen, the cardinality of $\mathscr{S}$ is $n\times (n-1)! =n!$ .
 We can then split $\det D$ into two terms 
\begin{align} \det D &=  \sum_{\sigma \in \mathfrak{S}_{n(n-1)}} \sgn(\sigma) \prod_{j=1}^{n(n-1)} D_{\sigma(j),j}\notag \\ 
											&= \sum_{\sigma \in \mathscr{S}} \sgn(\sigma) \prod_{j=1}^{n(n-1)} D_{\sigma(j),j} +\sum_{\sigma \notin \mathscr{S}} \sgn(\sigma) \prod_{j=1}^{n(n-1)} D_{\sigma(j),j} \label{decompDetD}
	\end{align}

Observe that considering 	$\sum_{\sigma \in \mathscr{S}} \sgn(\sigma) \prod_{j=1}^{n(n-1)} D_{\sigma(j),j}$ instead of $\det D$ is equivalent to considering the determinant of a new matrix $\tilde{D}$ with entries $\tilde{D}_{l,m} = D_{l,m}$ as soon as there is a permutation $\sigma \in \mathscr{S}$, such that $\sigma(m)=l$ and $\tilde{D}_{l,m}=0$ otherwise, that is 

\begin{multline}
\sum_{\sigma \in \mathfrak{S}} \sgn(\sigma) \prod_{j=1}^{n(n-1)} D_{\sigma(j),j} = \det \tilde{D}\\
=\det\left(\begin{array}{ccccccccccc}
\cline{1-7}
\multicolumn{1}{|c}{~}&~&~&\multicolumn{1}{c|}{~~}&B^{(0)}_{1,1}&~~&\multicolumn{1}{c|}{\mathbf{0}}&~~&~~&~~\\
\multicolumn{4}{|c|}{\mathbf{0}}
&~~&\ddots&\multicolumn{1}{c|}{~~}&
\multicolumn{4}{c}{\mathbf{0}}
~\\
\multicolumn{1}{|c}{~}&~&~&\multicolumn{1}{c|}{~~}&\mathbf{0}&~~&\multicolumn{1}{c|}{B^{(0)}_{(n-2),(n-2)}}&~~&~~&~~\\
\cline{1-7}
\multicolumn{1}{|c}{A^{(0)}_{(n-1),1}}&C^{(0)}_{(n-1),1}&\cdots&\multicolumn{1}{c|}{C^{(0)}_{(n-1),(n-1)}}&0&\cdots&0&\multicolumn{1}{|c}{~}&~&~~\\
\cline{1-7}
~~&~~&\vdots&~~&~~&~~&~~&\ddots&~~&~~\\
\cline{1-4}\cline{9-11}
\multicolumn{1}{|c}{~}&~&~&\multicolumn{1}{c|}{~~}&~~&~~&~~&~~&\multicolumn{1}{|c}{B^{(n-1)}_{1,1}}&~~&\multicolumn{1}{c|}{\mathbf{0}}\\
\multicolumn{4}{|c|}{\mathbf{0}}
&\multicolumn{4}{c}{\mathbf{0}}
&\multicolumn{1}{|c}{~~}&\ddots&\multicolumn{1}{c|}{~~}\\
\multicolumn{1}{|c}{~}&~&~&\multicolumn{1}{c|}{~~}&~~&~~&~~&~~&\multicolumn{1}{|c}{0}&~~&\multicolumn{1}{c|}{B^{(n-1)}_{(n-2),(n-2)}}\\
\cline{1-4}\cline{9-11}
\multicolumn{1}{|c}{A^{(n-1)}_{(n-1),1}}&C^{(n-1)}_{(n-1),1}&\cdots&\multicolumn{1}{c|}{C^{(n-1)}_{(n-1),(n-1)}}&~~&~~&~~&~~&\multicolumn{1}{|c}{0}&\cdots&\multicolumn{1}{c|}{0}\\
\cline{1-4}\cline{9-11}
\end{array}\right)
\end{multline}  
Now, rearranging the lines so that the lines $(A^{(i)}_{(n-1),1},C^{(i)}_{(n-1),1}, \dots,C^{(i)}_{(n-1),(n-1)}|{\mathbf 0})$ appear in the first $n$ lines of a block upper triangular matrix. Then, up to a possible minus sign
	\[ |\sum_{\sigma \in \mathfrak{S}} \sgn(\sigma) \prod_{j=1}^{n(n-1)} D_{\sigma(j),j} |= |\det(E)| \times |\det(F)| \]
	where $E$ is the $n\times n$-matrix defined by
	\begin{equation} E= 
	 \begin{pmatrix}
	    A^{(0)}_{(n-1),1}& C^{(0)}_{(n-1),1}& \cdots &C^{(0)}_{(n-1),(n-1)} \\
			\vdots & \vdots & ~~&\vdots \\
			A^{(n-1)}_{(n-1),1}& C^{(n-1)}_{(n-1),1}& \cdots &C^{(n-1)}_{(n-1),(n-1)} 
	 \end{pmatrix}
	\end{equation}
	and $F$ is a $n(n-2)\times n(n-2)$-diagonal matrix with diagonal entries 
	\[ ( B^{(0)}_{1,1}, \dots, B^{(0)}_{(n-2),(n-2)}, \dots,  B^{(n-1)}_{1,1}, \dots, B^{(n-1)}_{(n-2),(n-2)} ). \]
Using Lemma \ref{lem:coeffpresde0}, we get expansions of $\det E$ and $\det F$. From \eqref{coefB1}, we have
\begin{equation}\label{detF}
|\det F| = \prod_{i=0}^{n-1} \prod_{m=2}^{n-1} \left( r - \frac{ K_z^{(M)}(u_i, v^{(i)}_m)}{6} r^3 +o(r^3) \right).
\end{equation}
Then from \eqref{coefA1} and \eqref{coefC1}, we have 
\begin{equation}\label{detE}
   |\det E| = r^{n-1}|\det G|
\end{equation}
where 
\[ G = \begin{pmatrix}
         1- \frac{K_z^{(M)}(u_0,v)}2r^2 +o(r^2) & -u_0^{T} \left( 1- \frac{K_z^{(M)}(u_0,v)}6r^2 +o(r^2)\right) \\
				  \vdots & \vdots \\
				1- \frac{K_z^{(M)}(u_{n-1},v)}2r^2 +o(r^2) & -u_{n-1}^{T} \left( 1- \frac{K_z^{(M)}(u_{n-1},v)}6r^2 +o(r^2)\right).
       \end{pmatrix} 
	\]
In order to obtain the expansion of $|\sum_{\sigma \in \mathscr{S}} \sgn(\sigma) \prod_{j=1}^{n(n-1)} D_{\sigma(j),j} |$, it remains to expand $\det G$.
\begin{lem} \label{lem:detG} The determinant of $G$ has the following expansion
\[ \det G= (n-1)!\Delta_{n-1}(u_0, \dots, u_{n-1})( 1 - r^2\mathrm{Tr(\Delta^{-1}K)} +o(r^2) )\]
where 
\[\Delta=  \begin{pmatrix} 1& -u_0^{T} \\
																			\vdots &\vdots \\
																			1 & -u_{n-1}^{T}
																			\end{pmatrix}   \text{ and } K= \begin{pmatrix} \frac{K_z^{(M)}(u_0,v)}{2}& u_0^{T}\frac{K_z^{(M)}(u_0,v)}{6} \\
																			\vdots &\vdots \\
																			\frac{K_z^{(M)}(u_{n-1},v)}{2} & u_{n-1}^{T}\frac{K_z^{(M)}(u_{n-1},v)}{6}
																			\end{pmatrix}.   \]

\end{lem}

\begin{proof}[ Proof of Lemma \ref{lem:detG}]
The matrix $G$ can be written as $G= \Delta - r^2 K + o(r^2)$ with 
\[\Delta=  \begin{pmatrix} 1& -u_0^{T} \\
																			\vdots &\vdots \\
																			1 & -u_{n-1}^{T}
																			\end{pmatrix}   \text{ and } K= \begin{pmatrix} \frac{K_z^{(M)}(u_0,v)}{2}& u_0^{T}\frac{K_z^{(M)}(u_0,v)}{6} \\
																			\vdots &\vdots \\
																			\frac{K_z^{(M)}(u_{n-1},v)}{2} & u_{n-1}^{T}\frac{K_z^{(M)}(u_{n-1},v)}{6}
																			\end{pmatrix}   \]
Since, $\Delta$ is an invertible matrix for almost every $(u_0,\dots,u_{n-1})$, the expansion of the determinant function yields 
 \begin{equation}\label{eqDetG}
\det G = \det \Delta ( 1-r^2 \mathrm{Tr}(\Delta^{-1}K)) +o(r^2)
\end{equation}
where $\mathrm{Tr}$ denotes the trace operator. 
Now, observe that $\det \Delta$ is $(n-1)!$ times the volume of the $(n-1)$-dimensional simplex spanned by the vectors $-u_0,\dots,-u_{n-1}$ which has same volume as the simplex spanned by $u_0, \dots, u_{n-1}$ that is 
\begin{equation} \label{detDelta}\det \Delta  = (n-1)!\Delta(u_0, \dots, u_{n-1}) .\end{equation}
Inserting \eqref{detDelta} in \eqref{eqDetG}, we obtain the required expansion of $\det G$. 
 \end{proof}
By Lemma \ref{lem:detG}, \eqref{detF} and \eqref{detE}, we obtain 
\begin{align}
 |\sum_{\sigma \in \mathscr{S}}& \sgn(\sigma) \prod_{j=1}^{n(n-1)} D_{\sigma(j),j} |\notag \\
& = r^{n-1}(n-1)!\Delta(u_0, \dots, u_{n-1})( 1 - r^2 \mathrm{Tr(\Delta^{-1}K)} +o(r^2) ) \prod_{i=0}^{n-1}\prod_{m=2}^{n-1} (r- \frac{K_z^{(M)}(u_i,v_m^{(i)})}{6} r^3 +o(r^3) )\notag \\
&=  (n-1)! \Delta(u_0, \dots, u_{n-1}) \left( r^{n(n-1)-1} - r^{n(n-1)+1}( \sum_{i=0}^{n-1} \Ric_z^{v}(u_i) +\mathrm{Tr}( \Delta^{-1}K) ) +o( r^{n(n-1)+1}) \right) \label{detDsum1}
\end{align}
where $\Ric_{z}^{v}(u_i) = \sum_{m=2}^{n-1} K(u_i,v^{(i)}_m)=\Ric_{z}(u_i)-K_z^{(M)}(u_i,v)$. To conclude for the expansion of $\det D$, it suffices now to prove that 
the second sum in \eqref{decompDetD} has a negligeable contribution, that is
\begin{equation}\label{LimNeg} \lim_{r\to 0} \frac{1}{ r^{n(n-1)+1}} \sum_{\sigma \notin \mathscr{S}} \sgn(\sigma) \prod_{j=1}^{n(n-1)} D_{\sigma(j),j} =0 .\end{equation}
Indeed, let $\sigma \notin \mathscr{S}$. Then, there exists at least two indices $j$ such that $D_{\sigma(j),j}= A^{(i)}_{k,1}$ or $C^{(i)}_{k,m}$ or $B^{(i)}_{l,m}$ for some $k\neq n-1$, $l\neq m$. As in the proof of \eqref{dev}, because of Lemma \ref{lem:coeffpresde0}, the product $\prod_{j=1}^{n(n-1)} D_{\sigma(j),j}$ is at least of order $r^{n(n-1)+2}$ thus \eqref{LimNeg} holds. 
Combining \eqref{detDsum1} and \eqref{LimNeg}, we obtain 
\[ |\det D| =  (n-1)! \Delta(u_0, \dots, u_{n-1}) \left( r^{n(n-1)-1} - r^{n(n-1)+1}( \sum_{i=0}^{n-1} \Ric_z^{v}(u_i) +\mathrm{Tr}( \Delta^{-1}K) ) +o( r^{n(n-1)+1}) \right).\]
which inserted in \eqref{jacFin} concludes the proof of Proposition \ref{np}. 
\begin{flushright}
$\Box$
\end{flushright}

\section*{Concluding remarks}
 
\begin{enumerate}
\item General Blaschke-Petkantschin formula

Theorem \ref{BP} and Propositions \ref{deux} and \ref{np} provide asymptotic Blaschke-Petkantschin formulas for a $m$-tuple of points where $m=n+1$, 2 and $n$ respectively. Determining similar estimates for general $m$ is still open. The general method developed in this paper, goes along the following steps:\\
- the choice of suitable bases for the respective tangent spaces of the $m$ points and for the Grassmannian,\\
- the writing of the Jacobian matrix in function of several Jacobi fields,\\
- the expansion of each entry of the matrix with the Jacobi equation,\\
- the expansion of the Jacobian determinant by elementary row and column operations.\\

This could be applied in principle to any $m$. However, the practical difficulty comes from the calculations of the partial derivatives with respect to the vector space which belongs to the Grassmanniann. We would require a simple analytic way to describe the elements of the Grassmannian and rewrite the associated Haar measure, only it seems out of reach. For that matter, the classical Blachke-Petkantschin formula in the Euclidean space was not derived via a direct analytic calculation of the Jacobian but through the use of a classical Fubini-type exchange formula for the Haar measures on Grassmannians, see e.g. \cite[Theorem 7.1.1]{Wei08}, combined with an induction argument, see \cite[Theorem 7.2.1]{Wei08}. Moreover, when it becomes possible to make explicit the Haar measure as in the case $m=n$, we observe that it adds extra Jacobi fields as $\min(m,n+1-m)$ increases and that the final expansion of the determinant involves a more and more intricate second term which does not seem to have a simple geometrical meaning.

\item Uniformity of the expansions

When $M$ is a compact Riemannian manifold, it is possible to show that the expansion of the Jacobian determinant provided by \eqref{dev} is uniform with respect to $z, u_0,u_1,\dots,u_n$. To do so, we need the two following steps:\\
- check carefully that each remainder in the expansions of the coefficients in Lemma \ref{diag} is uniform,\\
- show that this uniformity of the remainder is preserved during the calculation of the determinant.

The first step requires in particular to rewrite a Taylor expansion with a Lagrange-type remainder for each scalar product involving a Jacobi field.

The uniformity of the Jacobian expansion plays a key role in the on-going work \cite{CCElimite} on mean asymptotics for a Poisson-Voronoi tessellation in $M$, when  applying a Blaschke-Petkantschin change of variables and integrating over $u_0,\dots,u_n$ and/or over $z$.  

\end{enumerate}
\noindent{\bf Acknowledgements.} The author thanks Pierre Calka and Nathana\"el Enriquez, her PhD advisors, for helpful discussions, comments and suggestions. This work was partially supported by the French ANR grant PRESAGE (ANR-11-BS02-003), the French research group GeoSto (CNRS-GDR3477) and the Institut Universitaire de France.

\bibliography{BiblioVoronoi}

\begin{thebibliography}{CCE18}

\bibitem[Ber03]{Ber03}
M.~Berger.
\newblock {\em A panoramic view of {R}iemannian geometry}.
\newblock Springer-Verlag, Berlin, 2003.

\bibitem[Bla35]{Bla35}
W.~Blaschke.
\newblock Integralgeometrie 1. {E}rmittlung der {D}ichten f{\"u}r linear
  {U}nterr{\"a}ume im {E}n.
\newblock {\em Actualit{\'e}s Scientifiques et Industrielles}, 252, 1935.

\bibitem[CCE]{CCElimite}
P.~Calka, A.~Chapron, and N.~Enriquez.
\newblock Limit theorems, variance asymptotics and estimation for
  {P}oisson-{V}oronoi tessellations in a {R}iemannian manifold.
\newblock in preparation.

\bibitem[CCE18]{AurelieArticle}
P.~Calka, A.~Chapron, and N.~Enriquez.
\newblock Mean asymptotics for a {P}oisson-{V}oronoi cell in a {R}iemannian
  manifold.
\newblock submitted, 2018.

\bibitem[DC92]{Do92}
M.P. Do~Carmo.
\newblock {\em {R}iemannian geometry}.
\newblock Mathematics: Theory \& Applications. Birkh\"auser Boston, Inc.,
  Boston, MA, 1992.
\newblock Translated from the second Portuguese edition by Francis Flaherty.

\bibitem[Iso00a]{Iso00b}
Y.~Isokawa.
\newblock {P}oisson-{V}oronoi tessellations in three-dimensional hyperbolic
  spaces.
\newblock {\em Adv. in Appl. Probab.}, 32(3):648--662, 2000.

\bibitem[Iso00b]{Iso2000}
Y.~Isokawa.
\newblock Some mean characteristics of {P}oisson-{V}oronoi and
  {P}oisson-{D}elaunay tessellations in hyperbolic planes.
\newblock {\em Bull. Fac. Ed. Kagoshima Univ. Natur. Sci.}, 52:11--25 (2001),
  2000.

\bibitem[Lee97]{Lee97}
J.M. Lee.
\newblock {\em {R}iemannian manifolds}, volume 176 of {\em Graduate Texts in
  Mathematics}.
\newblock Springer-Verlag, New York, 1997.
\newblock An introduction to curvature.

\bibitem[Mil70]{Mil70}
R.E. Miles.
\newblock A synopsis of ``{P}oisson flats in {E}uclidean spaces''.
\newblock {\em Izv. Akad. Nauk Armjan. SSR Ser. Mat.}, 5(3):263--285, 1970.

\bibitem[Mil71a]{Mil71bis}
R.E. Miles.
\newblock Isotropic random simplices.
\newblock {\em Advances in Appl. Probability}, 3:353--382, 1971.

\bibitem[Mil71b]{Mil71}
R.E. Miles.
\newblock Random points, sets and tessellations on the surface of a sphere.
\newblock {\em Sankhy\=a Ser. A}, 33:145--174, 1971.

\bibitem[M{\o}l94]{Mol94}
J.~M{\o}ller.
\newblock {\em Lectures on random {V}orono\u\i\ tessellations}, volume~87 of
  {\em Lecture Notes in Statistics}.
\newblock Springer-Verlag, New York, 1994.

\bibitem[Pet35]{Pet35}
B.~Petkantschin.
\newblock Integralgeometrie 6. {Z}usammenh\"ange zwischen den {D}ichten der
  linearen {U}nterr\"aume im $n$- dimensionalen {R}aum.
\newblock {\em Abh. Math. Sem. Univ. Hamburg}, 11(1):249--310, 1935.

\bibitem[SW08]{Wei08}
R.~Schneider and W.~Weil.
\newblock {\em Stochastic and integral geometry}.
\newblock Probability and its Applications (New York). Springer-Verlag, Berlin,
  2008.

\end{thebibliography}
\bibliographystyle{alpha}
\end{document}